\documentclass[11pt]{amsart}
\usepackage[hypertex]{hyperref}
\usepackage{amsmath,amsthm,amssymb, amscd, amsfonts}
\usepackage[all]{xy}
\usepackage{leftidx}
\usepackage[latin1]{inputenc}
\usepackage{enumerate}

\theoremstyle{plain}

\newtheorem{theorem}{Theorem}[section]
\newtheorem{corollary}[theorem]{Corollary}

\newtheorem{proposition}[theorem]{Proposition}
\newtheorem{lemma}[theorem]{Lemma}

\theoremstyle{definition}

\theoremstyle{remark}
\newtheorem{remark}[theorem]{Remark}

\numberwithin{equation}{section}\theoremstyle{plain}

\newcommand{\I}{\mathcal{I}}

\renewcommand{\1}{\textbf{1}}

\newcommand{\A}{{\mathcal A}}
\newcommand{\B}{{\mathcal B}}
\newcommand{\C}{{\mathcal C}}
\newcommand{\D}{{\mathcal D}}
\newcommand{\Ll}{{\mathcal L}}

\newcommand{\Z}{{\mathcal Z}}

\newcommand{\M}{\mathcal{M}}

\newcommand{\E}{{\mathcal E}}

\newcommand{\W}{{\mathcal W}}

\newcommand{\Rep}{\operatorname{Rep}}

\newcommand\Aut{\operatorname{Aut}}
\newcommand\Irr{\operatorname{Irr}}
\newcommand\FPdim{\operatorname{FPdim}}

\newcommand\vect{\operatorname{Vec}}
\newcommand\svect{\operatorname{sVec}}
\newcommand\id{\operatorname{id}}
\newcommand\ad{\operatorname{ad}}

\newcommand\rev{\operatorname{rev}}

\newcommand\Hom{\operatorname{Hom}}

\begin{document}
\title[Non-degenerate braided fusion categories]{On weakly group-theoretical
non-degenerate braided fusion categories}
\author{Sonia Natale}
\address{Facultad de Matem\'atica, Astronom\'\i a y F\'\i sica.
Universidad Nacional de C\'ordoba. CIEM -- CONICET. Ciudad
Universitaria. (5000) C\'ordoba, Argentina}
\email{natale@famaf.unc.edu.ar
\newline \indent \emph{URL:}\/ http://www.famaf.unc.edu.ar/$\sim$natale}

\thanks{Partially supported by  CONICET and SeCYT--UNC}



\date{\today}

\begin{abstract} We show that the Witt class of a weakly group-theoretical
non-degenerate braided fusion category belongs 
to the subgroup generated by classes of non-degenerate pointed braided fusion
categories and Ising braided categories. This applies in particular to solvable
non-degenerate braided fusion categories.
We also give some sufficient conditions for a braided fusion category to be
weakly group-theoretical or solvable in terms of the
factorization of its Frobenius-Perron dimension and the Frobenius-Perron
dimensions of its simple objects.
As an application, we prove that every non-degenerate braided fusion category
whose Frobenius-Perron dimension is a natural number less than $1800$, or an odd
natural number less than $33075$, is
weakly group-theoretical.
   \end{abstract}

\maketitle

\section{Introduction}

A fusion category $\C$ is called \emph{weakly group-theoretical} if it is
categorically Morita equivalent to a nilpotent fusion category, that is, if
there exists
an indecomposable module category $\M$ such that $\C^*_{\M}$ is a nilpotent
fusion category. In particular, every weakly group-theoretical fusion category
has integer Frobenius-Perron
dimension. If,   furthermore, $\C$ is Morita equivalent to a cyclically
nilpotent fusion category, then $\C$ is called \emph{solvable}. 
Equivalently, $\C$ is  solvable if there exist sequences $\vect = \C_0, \dots,
\C_n = \C$, of fusion categories, and $G_1, \dots, G_n$, of cyclic groups of
prime order  such that for all $1 \leq i \leq n$, $\C_i$ is a
$G_i$-equivariantization or a $G_i$-extension of $\C_{i-1}$.
We refer the reader to \cite{ENO2} for other characterizations and main
properties of weakly group-theoretical and related fusion categories. 

\medbreak 
An important class of fusion categories is that of braided fusion categories,
that is, fusion categories $\C$ endowed with natural isomorphisms $c: X \otimes
Y \to Y \otimes X$, $X, Y \in \C$, called a \emph{braiding}, subject to
appropriate axioms. 
Two extreme classes of braided fusion categories, so-called symmetric and
non-degenerate braided fusion categories, appear related to the square of the
braiding.  

Symmetric fusion categories have been classified by Deligne \cite{deligne}. On
the other side, a number of important results concerning the structure of a
non-degenerate braided fusion category have been established in the literature.
A non-degenerate braided fusion category endowed with a compatible
ribbon  structure is called a \emph{modular} category. Modular
categories have many applications in
distinct areas of mathematics and mathematical physics, for instance, in
low-dimensional topology, they constitute an important tool in
the construction of invariants of knots and $3$-manifolds. See \textit{e.g.}
\cite{BK, turaev-b}.

\medbreak The group of Witt classes of non-degenerate braided
fusion categories, denoted $\W$, was introduced in \cite{witt-nondeg}. 
Two non-degenerate braided fusion categories $\C_1$ and $\C_2$ are called
\emph{Witt equivalent} if there exist fusion categories $\D_1$ and $\D_2$ such
that $\C_1 \boxtimes \Z(\D_1) \cong \C_2 \boxtimes \Z(\D_2)$ as braided tensor
categories, where $\Z(\D_i)$ denotes the Drinfeld center of the fusion category
$\D_i$, $i = 1, 2$. 

The Witt group $\W$ consists of equivalence classes of
non-degenerate brai\-ded fusion categories under this equivalence relation with
multiplication induced by Deligne's tensor product $\boxtimes$. The unit
element is the class of the category $\vect$ of finite-dimensional vector spaces
over the base field $k$ and the inverse of the class of a non-degenerate braided
fusion category  $\C$ is the class of the reverse braided fusion category
$\C^{\rev}$. 
This endows $\W$ with the structure of an (infinite countable) abelian group.

The explicit determination of the structure of the group $\W$ and the relations
amongst its elements are pointed out in \cite{witt-nondeg} as relevant problems
in connection with the classification 
of fusion categories.

\medbreak Let $\W_{pt}$ and $\W_{Ising}$
denote, respectively, the subgroup of Witt classes of pointed non-degenerate
fusion categories and the subgroup generated by Witt classes of Ising braided
categories. 

Recall that an \emph{Ising braided category} is a  non-pointed braided fusion
category of Frobenius-Perron dimension $4$.   Ising braided categories were
classified in \cite[Appendix B]{DGNOI}; it is known that they fall into $8$
equivalence classes and all of them are non-degenerate. If $\I$ is an Ising
braided category, then the pointed subcategory $\I_{pt}$
is the unique nontrivial (symmetric) subcategory of $\I$, and it is equivalent to
the category $\svect$ of super-vector spaces. Besides, $\I$ has a unique
non-invertible simple object of 
Frobenius-Perron dimension $\sqrt 2$.  

\medbreak The subgroups $\W_{pt}$ and $\W_{Ising}$ are explicitly described in
\cite[Appendix A.7 and Appendix B]{DGNOI};  see also \cite[Sections 5.3 and 6.4
(3)]{witt-nondeg}. We have that $\W_{Ising}$ is
isomorphic to the cyclic group of
order $16$. On the other hand, if $\W_{pt}(p)$ denotes the group of classes of
metric $p$-groups, we have an isomorphism  $\W_{pt} \cong \bigoplus_{p \text{
prime }}
\W_{pt}(p)$. In addition,  $\W_{pt}(2) \simeq \mathbb Z_8 \oplus \mathbb Z_2$, 
$\W_{pt}(p) \simeq \mathbb Z_4$, if $p = 3 (\text{mod } 4)$, and $\W_{pt}(p)
\simeq \mathbb Z_2 \oplus \mathbb Z_2$, if $p = 1 (\text{mod } 4)$. 

\medbreak In this paper we show that if $\C$ is a non-degenerate braided fusion
category such that $\C$ is weakly group-theoretical, then the Witt class $[\C]$
of $\C$ belongs to the subgroup generated by $\W_{pt}$ and $\W_{Ising}$. If,
moreover, $\C$ is integral, then $[\C] \in \W_{pt}$. See Theorem \ref{wgt-witt}.

The proof of Theorem \ref{wgt-witt} is given in Section \ref{wgt}. It relies on
results of the paper \cite{witt-nondeg}. It makes use as well of  
the notion of a braided group-crossed fusion category introduced by Turaev 
\cite{turaev, turaev2} and its main properties, in particular, its connection
with the existence of nontrivial Tannakian subcategories in a braided fusion
category. These results are recalled in Sections \ref{tens-c} and
\ref{g-crossed}. Using these tools, we also prove in Section
\ref{solvable} a related result (Theorem \ref{sol-non-deg}) on the structure of
solvable braided fusion
categories. 

\medbreak Let $\widetilde\W$ be the subgroup of $\W$ generated by Witt
equivalence
classes of  the fusion categories $\C(\mathfrak g, l)$ of
integrable highest weight modules of level $l \in \mathbb Z_+$ over the
affinization of a simple finite-dimensional Lie algebra $\mathfrak g$. 
It is shown in \cite{witt-nondeg} that $\W_{pt}, \W_{Ising} \subseteq
\widetilde\W$.
Conjecturally, $\widetilde\W$ coincides with the subgroup $\W_{un}$ of Witt
classes of pseudo-unitary non-degenerate braided fusion categories \cite[Question
6.4]{witt-nondeg}. 

On the other side, it is also conjectured that every fusion category of integer
Frobenius-Perron dimension is weakly group-theoretical \cite[Question 2]{ENO2}.
As a consequence of Theorem \ref{wgt-witt}, we obtain that for every 
non-degenerate braided fusion category $\C$ such that $\C$ is weakly
group-theoretical, then $[\C] \in \widetilde\W$. 

\medbreak One of the main results of \cite{ENO2} establishes the analogue of
Burnside's
$p^aq^b$-theorem for fusion categories, namely, that any fusion category $\C$
whose Frobenius-Perron dimension is $p^aq^b$,  where $p$ and $q$ are prime
numbers and $a, b$ are non-negative integers, is solvable.
Some solvability results for braided fusion categories have been obtained in
\cite{char-deg, cd2-wint}. In particular, if $\C$ is a braided fusion category 
such that the Frobenius-Perron dimensions of simple objects of
$\C$ are $\leq 2$, or if $\FPdim \C$ is odd and the Frobenius-Perron dimensions
of simple objects of $\C$ are powers of a fixed prime number, then $\C$ is
solvable.

\medbreak Combining the main properties of braided group-crossed fusion
categories with
the
methods developed in the paper \cite{ENO2}, we also
give in Section \ref{suff-cond} some further sufficient conditions for a braided
fusion
category to be
solvable or weakly group-theoretical. We show that every weakly
integral braided fusion category whose Frobenius-Perron dimensions of simple
objects are powers of a fixed prime number is always solvable. See Theorem
\ref{fpdim-pn}. This extends the previously mentioned results in \cite{char-deg,
cd2-wint}. 

In addition, we show that every non-degenerate braided fusion category
$\C$ whose Frobenius-Perron dimension factorizes in the form $\FPdim \C =
p^aq^bc$,
where $p$ and $q$ are prime numbers, $a$ and $b$ are nonnegative integers, and $c$ is a square-free
natural number, is necessarily weakly group-theoretical. See Theorem
\ref{fact-wgt}. 

In Section \ref{low-dim} we apply this result  to
show in Theorems \ref{less1800} and \ref{oddless33075}, respectively, that every
weakly integral non-degenerate braided
fusion category of Frobenius-Perron dimension less than $1800$ is weakly
group-theoretical and moreover, it is solvable if its  Frobenius-Perron
dimension is odd and less than $33075$.
   
\subsection*{Acknowledgement} This paper was written during a visit to the
Institute des Hautes \' Etudes Scientifiques, France. The author is grateful to
the IHES for the outstanding hospitality and the excellent atmosphere.

\section{Preliminaries and notation} 

We shall work over an algebraically closed field $k$ of characteristic zero. The
category of finite dimensional vector spaces
over $k$ will be denoted by $\vect$. A fusion category over $k$
is a semisimple rigid monoidal category over $k$ with finitely many isomorphism
classes of simple objects, finite-dimensional Hom spaces, and such that the unit
object $\1$ is simple.  We
refer the reader to \cite{ENO, ENO2, DGNOI} for the main notions about fusion
categories and braided fusion categories used throughout.
Unless otherwise stated, all tensor categories will be assumed to be strict.

\subsection{Frobenius-Perron dimensions} Let $\C$ be a fusion category. The
Frobenius-Perron dimension of a
simple object $X \in \C$ is, by definition, the Frobenius-Perron eigenvalue of
the matrix of left multiplication by the class of $X$ in the basis $\Irr(\C)$ of
the Grothendieck ring of $\C$ consisting of isomorphism classes of simple
objects. The Frobenius-Perron dimension of $\C$ is the number $\FPdim \C =
\sum_{X \in \Irr(\C)} (\FPdim X)^2$. The category $\C$ is called integral  if
$\FPdim X \in \mathbb Z$, for all simple object $X \in \C$, and it is called
weakly integral if $\FPdim \C \in \mathbb Z$. 

If $\C$ is a weakly integral fusion category, then $(\FPdim X)^2 \in \mathbb Z$, for
all simple object $X \in \C$ \cite[Proposition 8.27]{ENO}. This implies, in
particular, that a fusion subcategory of $\C$ is also weakly integral. On the
other hand, if $\C$ is weakly integral (respectively, integral) and $F: \C \to
\D$ is a dominant (or surjective) tensor functor, then $\D$ is weakly integral (respectively, integral)
as well; see \cite[Proposition 2.12]{indp-exact}, \cite[Corollary 8.36]{ENO}.

\subsection{Nilpotent and weakly group-theoretical fusion categories} Let $G$ be
a finite group. A $G$-grading on a fusion category $\C$ is a decomposition $\C = \oplus_{g\in G} \C_g$, such that $\C_g \otimes \C_h \subseteq \C_{gh}$ and $\C_g^* \subseteq \C_{g^{-1}}$, for all $g, h \in G$. 
The fusion category $\C$ is called a \emph{$G$-extension} of a
fusion category $\D$ if there is a faithful grading $\C = \oplus_{g\in G} \C_g$
with neutral component $\C_e \cong \D$. 

If $\C$ is any fusion category, there exist a finite group $U(\C)$, called the
\emph{universal grading group} of $\C$, and a canonical faithful grading $\C =
\oplus_{g \in U(\C)}\C_g$, with neutral component $\C_e = \C_{\ad}$, where
$\C_{\ad}$ is the adjoint
subcategory of $\C$, that is, the fusion subcategory generated by $X\otimes X^*$, $X \in \Irr(\C)$.

A fusion category $\C$ is (cyclically) \emph{nilpotent} if there exists a
sequence of fusion categories $\vect = \C_0 \subseteq \C_1 \dots \subseteq \C_n = \C$, and finite
(cyclic) groups $G_1, \dots, G_n$, such that for all $i = 1, \dots, n$, $\C_i$
is a $G_i$-extension of $\C_{i-1}$.

\medbreak Dual to the notion of a group extension, we have the notion of an
equivariantization. Consider  an action of a finite group $G$ on a fusion
category $\C$ by tensor autoequivalences $\rho: \underline
G \to \underline \Aut_{\otimes} \, \C$.
The \emph{equivariantization} of $\C$ with respect to the action $\rho$, denoted
$\C^G$, is a fusion category whose objects are pairs  $(X, \mu)$, such that $X$
is an object of $\C$ and $\mu = (\mu^g)_{g \in G}$, is a collection of
isomorphisms $\mu^g:\rho^gX \to X$, $g \in G$, satisfying appropriate
compatibility conditions.  

The forgetful functor $F: \C^G \to \C$, $F(X, \mu) = X$,
is a dominant tensor functor that gives rise to a central exact sequence of
fusion categories $\Rep G \to \C^G \to \C$ \cite{indp-exact}, where $\Rep G$ is
the category of finite-dimensional representations of $G$.

\medbreak Two fusion categories $\C$ and $\D$
are \emph{Morita equivalent} if $\D$ is equivalent to the dual
$\C^*_{\mathcal M}$ with respect to an indecomposable module
category $\mathcal M$.

A  fusion category $\C$ is called \emph{weakly group-theoretical} (respectively,
\emph{solvable}) if it is Morita equivalent to a nilpotent (respectively,
cyclically nilpotent) fusion category. 

It is shown in \cite[Proposition 4.1]{ENO2} that the class of weakly
group-theoreti\-cal fusion categories is stable under the operations of taking 
extensions, equivariantizations, Morita equivalent categories, tensor products,
Drinfeld center, fusion subcategories and components of quotient categories.
Also, the class of solvable fusion categories is stable under taking extensions
and equivariantizations
by solvable groups, Morita equivalent categories, tensor products, Drinfeld
center, fusion subcategories and components of quotient categories.

\subsection{Braided fusion categories} A braiding in a fusion category $\C$ is a
 natural isomorphism
$c_{X,Y} : X \otimes Y \rightarrow Y \otimes X$, $X, Y \in \C$, subject to the
hexagon
axioms. A braided fusion category is a fusion category endowed with a braiding. 

Suppose $\C$ is a braided fusion category. The reverse braided fusion
category
will be denote by $\C^{\rev}$; thus, if $c_{X,
Y}: X \otimes Y \to Y \otimes X$ denotes the braiding of $\C$, then $\C^{\rev} =
\C$ as a fusion category, with braiding $c^{\rev}_{X, Y} = c_{Y, X}^{-1}$, for
all objects $X, Y$. 

\medbreak If $\D$ is a fusion subcategory of a braided fusion category $\C$,
the M\" uger centralizer of $\D$ in $\C$ will be denoted by $\Z_2(\D, \C)$, or
also by $\D'$ when there is no ambiguity. Thus $\Z_2(\D, \C)$ is the full fusion
subcategory generated by all objects $X \in \C$ such that $c_{Y, X}c_{X, Y} =
\id_{X \otimes Y}$, for all objects $Y \in \D$.

\medbreak The M\" uger (or symmetric) center of $\C$ will be denoted by
$\Z_2(\C) : = \Z_2(\C, \C)$. 
The category $\C$ is called \emph{symmetric} if $\Z_2(\C) = \C$. If $\C$ is any
braided fusion category, its M\" uger center $\Z_2(\C)$ is a symmetric fusion
subcategory of $\C$.
On the opposite extreme,  $\C$ is  called 
\emph{non-degenerate} (respectively, \emph{slightly degenerate})
if $\Z_2(\C) \cong \vect$ (respectively, if $\Z_2(\C) \cong \svect$).

\medbreak For a fusion category $\C$, the Drinfeld center of
$\C$ will be
denoted $\Z(\C)$. It is known that $\Z(\C)$ is a braided non-degenerate fusion
category
of Frobenius-Perron dimension $\FPdim \Z(\C) = (\FPdim \C)^2$. Necessary and
sufficient conditions for a braided fusion category to be equivalent to the
center of some fusion category are given in \cite{witt-nondeg}.

\medbreak
Let $G$ be a finite group. The fusion category of
finite dimensional representations of $G$ will be denoted by $\Rep G$. This is a
symmetric fusion category with respect to the canonical braiding.
A braided fusion category $\E$ is called Tannakian, if $\E \cong \Rep G$ for
some finite group $G$ as symmetric fusion categories.

A  Theorem of Deligne \cite{deligne}, states that every
symmetric fusion
category $\Ll$ is super-Tannakian, meaning that there exist a finite group $G$
and a central element $u \in G$ of order $2$, such that $\Ll$ is equivalent to
the category $\Rep(G, u)$ of representations of $G$ on finite-dimensional
super-vector spaces where $u$ acts as the parity operator.  

\medbreak Hence if $\Ll \cong
\Rep(G, u)$ is a symmetric fusion category, then $\E = \Rep G/u$ is a
Tannakian subcategory of $\Ll$ and $\FPdim \E = \FPdim \Ll / 2$; in
particular, if $\FPdim \Ll > 2$, then $\Ll$ necessarily contains a Tannakian
subcategory, and a non-Tannakian symmetric fusion category of Frobenius-Perron
dimension $2$ is equivalent to the category $\svect$ of finite-dimensional
super-vector spaces. See
\cite[Section 2.4]{ENO2}.

\section{Connected \' etale algebras in braided fusion categories}\label{tens-c}

Let $\C$ be a braided fusion category. Recall from \cite{witt-nondeg} that a
separable commutative algebra $A \in \C$ is called an \emph{\' etale algebra} in
$\C$. If $\Hom_\C(\1, A) \cong k$, then $A$ is called \emph{connected}.

Let $A \in \C$ be a connected \' etale algebra. Let also $\C_A$ denote the
category of right $A$-modules in $\C$ and $\C_A^0$ the category of dyslectic
$A$-modules. If $\C$ is a non-degenerate braided fusion category, then there is
an equivalence of braided fusion categories 
\begin{equation}\label{equiv} \C \boxtimes (\C_A^0)^{\rev} \cong \Z(\C_A),
\end{equation}
such that the restriction of the forgetful functor $U: \Z(\C_A) \to \C_A$ to $\C
\cong \C \boxtimes \vect$ is isomorphic to the free module functor $F_A: \C \to
\C_A$, $F_A(X) = X \otimes A$. 
See \cite[Corollary 3.30 and Remark 3.31 (i)]{witt-nondeg}. It follows from this
that $\C_A^0$ is a
non-degenerate braided fusion category and $\FPdim \C_A^0 = \FPdim \C / (\FPdim
A)^2$. Moreover, $(\C_A^0)^{\rev} \simeq \Z_2(\C, \Z(\C_A))$ as braided fusion
categories.

\begin{proposition}\label{cor-etale} Let $\C$ be a non-degenerate braided fusion
category. Suppose
that $A \in \C$ is a connected \' etale algebra. Then $\C$ is weakly integral 
(respectively, integral, weakly group-theoretical, solvable or
group-theoretical) if and only if $\C_A$ is weakly integral  (respectively,
integral, weakly group-theoretical, solvable or group-theoretical).
\end{proposition}

\begin{proof} Observe that there is a dominant tensor functor $F:\C \to \C_A$.
This implies the 'only if' direction. Suppose now that $\C_A$ is in one of the
prescribed classes, that is, it is weakly integral, integral, weakly
group-theoretical, solvable or group-theoretical. Then the center of $\C_A$ is
in the same class and, because by \eqref{equiv}, $\C$ is equivalent to a fusion
subcategory of $\Z(\C_A)$, then $\C$ is in that class as well. \end{proof}

For  a non-degenerate fusion category $\C$, we shall denote by $[\C]$ its
equivalence class in the Witt group. Recall from \cite[Corollary 5.9]{witt-nondeg} that two
non-degenerate braided fusion categories $\C_1$ and $\C_2$ are Witt-equivalent
if and only if there exists a fusion category $\D$ such that $\Z(\D) \cong \C_1
\boxtimes \C_2^{\rev}$ as braided fusion categories.

In view of the equivalence \eqref{equiv}, if $A \in \C$ is a connected \' etale
algebra, then $[\C] = [\C_A^0]$.

\section{Braided fusion categories and braided $G$-crossed fusion
categories}\label{g-crossed}

Let $G$ be a finite group. Recall that a \emph{braided $G$-crossed fusion
category}
\cite{turaev, turaev2} is a fusion category $\A$ endowed with a $G$-grading $\A
= \oplus_{g \in G}\A_g$ and an action of $G$ by tensor autoequivalences
$\rho:\underline G \to \underline \Aut_{\otimes} \, \A$, such that $\rho^g(\A_h) \subseteq
\A_{ghg^{-1}}$, for all $g, h \in G$, and a $G$-braiding $c: X \otimes Y \to
\rho^g(Y) \otimes X$, $g \in G$, $X \in \A_g$, $Y \in \A$, subject to compatibility conditions.

\medbreak A Tannakian subcategory $\E$ of a braided fusion category $\C$ gives
rise to a
connected \' etale algebra $A$ in $\C$. If $G$ is a finite group such that $\E
\cong \Rep G$ as symmetric categories, then $A$ is the algebra of functions on
$G$ with the regular action of $G$.

The fusion category $\C_A$ is in this case the de-equivariantization $\C_G$ of
$\C$ with respect to $\Rep G$, and it is a braided $G$-crossed fusion category.
 
The braided fusion category $\C_A^0$ is the neutral
component of $\C_G$ with respect to the associated $G$-grading.

\medbreak Conversely, let $\A$ be a $G$-crossed braided fusion category. Then
the
equivariantization $\A^G$ under the action of $G$ is a braided fusion category.
The canonical embedding $\Rep G \to \A^G$ of fusion categories is fact an
embedding of braided fusion categories. Hence $\A^G$ contains $\E \cong \Rep G$
as a Tannakian subcategory. 

The $G$-braiding on $\A$ restricts to a braiding in the neutral component
$\A_e$ of the $G$-grading. Furthermore, the group $G$ acts by restriction on
$\A_e$ and this action is by braided tensor autoequivalences. This makes the
equivariantization $\A_e^G$ into a braided fusion subcategory of $\A^G$. This
fusion subcategory coincides with the centralizer $\Z_2(\E, \A^G)$ of the
Tannakian subcategory $\E$ in $\A^G$. See  \cite{mueger-crossed}. 

\medbreak In this way, equivariantization defines a bijective correspondence
between
equivalence classes of braided fusion categories containing $\Rep G$ as a
Tannakian subcategory and $G$-crossed braided fusion categories \cite{kirillov},
\cite{mueger-crossed}, \cite[Section 4.4]{DGNOI}.
The braided fusion category $\A^G$ is non-degenerate if and only if
the neutral component $\A_e$ is non-degenerate and the $G$-grading of $\A$ is
faithful \cite[Proposition 4.6 (ii)]{DGNOI}.

In particular, if $\C$ is a non-degenerate braided fusion category
containing a Tannakian subcategory $\E \cong \Rep G$, then $|G|^2$ divides
$\FPdim \C$.

\medbreak Let $\A$ be a $G$-crossed braided fusion category such that the
neutral component $\A_e$ is non-degenerate and the $G$-grading of $\A$ is
faithful.  As a consequence of \eqref{equiv}, we have $\A_e^{\rev} = \Z_2(\A^G,
\Z(\A))$ and there is an equivalence
of braided fusion categories 
\begin{equation}\label{center} \Z(\A) \simeq \A^G \boxtimes \A_e^{\rev}.
\end{equation}

In this context we have the following refinement of Proposition
\ref{cor-etale}: 

\begin{proposition}\label{tann-inher} Let $\C$ be a braided fusion category.
Suppose
that $\E \cong \Rep G \subseteq  \C$ is a Tannakian subcategory. Then $\C$ is
weakly integral 
(respectively, integral or weakly group-theoretical) if and only if
$\C^0_G$ is weakly integral  (respectively,
integral, weakly group-theoretical). In addition,  $\C$ is solvable if and only
if $\C^0_G$ is solvable and $G$ is solvable.
\end{proposition}

\begin{proof} The statement concerning weakly integral, integral and weakly
group-theoretical fusion categories follows from Proposition \ref{cor-etale},
since the de-equi\-variantization $\C_G$ is an $H$-extension of $\C_G^0$, for a (normal)
subgroup $H$ of $G$. Suppose that $\C$  is solvable. Then the quotient category
$\C_G$ and hence its fusion subcategory $\C_G^0$ are solvable as well. Moreover,
since $\C \cong (\C_G)^G$ is a $G$-equivariantization, then the category $\Rep
G$ is equivalent to a fusion subcategory of $\C$ and it is thus solvable. Hence
the group $G$ is solvable.  If, on the other hand, $\C^0_G$ and $G$ are
solvable, then $\C_G$ is solvable because it is an $H$-extension of $\C_G^0$, for some subgroup $H \subseteq G$.
Hence so is $\C \cong (\C_G)^G$.  \end{proof}

Note that $\C$ is obtained from $\C^0_G$ by an $H$-extension, where $H \subseteq
G$ is a subgroup of $G$ (the support of $\C_G$) followed by a
$G$-equivariantization.
Since the class of group-theoretical fusion categories is not stable under the
operation of taking extensions, then the property of being group-theoretical is
not inherited in general from
$\C^0_G$. 

\section{Solvable braided fusion categories}\label{solvable}

Recall that a fusion category $\C$ is called group-theoretical  if $\C$ is
Morita equivalent to a pointed fusion category. 

\medbreak If $\C$ is a braided fusion category, it is shown in \cite[Theorem
7.2]{NNW}
that $\C$ is group-theoretical if and only if $\C$ contains a Tannakian
subcategory $\E \cong \Rep G$ such that the de-equivariantization $\C_G$ is a
pointed fusion category.
This immediately implies the following: 

\begin{lemma}\label{gt} Let $\C$ be a group-theoretical braided fusion category.
Then either $\C$ is pointed or it contains a nontrivial Tannakian subcategory. 
\end{lemma}

\begin{proposition}\label{solv-tann} Let $\C$ be a braided solvable fusion
category. Assume in addition that $\C$ is integral. Then either $\C$ is pointed
or it contains a nontrivial Tannakian subcategory. 
\end{proposition}

\begin{proof} Since $\C$ is solvable, there exist a group $G$ of prime order and
a fusion category $\D$ such that $\C$ is equivalent as a fusion category to a
$G$-equivariantization or to a $G$-extension of $\D$. In particular, $\D$ is
integral and solvable and $\FPdim \D = \FPdim \C / |G| < \FPdim \C$.

If $\C$ is a $G$-equivariantization of $\D$, then there is a central exact
sequence of
tensor functors $\Rep G \to \C \to \D$ and $\Rep G$ is a
Tannakian subcategory of $\C$. See \cite[Example 2.5 and Proposition
2.6]{indp-exact}.

Suppose next that $\C$ is a $G$-extension of $\D$, then $\D$ is a braided fusion
category and we may assume inductively that $\D$ contains a Tannakian
subcategory, whence so does $\C$, or $\D$ is pointed.  The last possibility
implies that $\C$ is nilpotent. By \cite[Theorem 6.10]{DGNO-gt}, an integral
nilpotent braided fusion category is group-theoretical. Then Lemma \ref{gt}
implies that $\C$ has one of the required properties. \end{proof}

\begin{theorem}\label{sol-non-deg} Let $\C$ be a solvable non-degenerate braided
 fusion category. Then one of the following holds:
\begin{enumerate}\item[(i)] $\C$ contains a nontrivial Tannakian subcategory, or
\item[(ii)] $\C \cong \mathcal B \boxtimes \I_1 \boxtimes \dots \boxtimes \I_n$,
as braided fusion categories, where $\B$ is a pointed non-degenerate fusion
category and $\I_1, \dots, \I_n$ are Ising braided categories.  
\end{enumerate}
\end{theorem}

\begin{proof} The proof is by induction on $\FPdim \C$ (note that, since $\C$ is
solvable, $\FPdim \C$ is a natural integer).
In view of Proposition \ref{solv-tann}, we may assume that $\C$ is not integral.
We may further assume that $\C$ is prime, that is,  $\C$ contains no proper
non-degenerate fusion subcategories other than $\vect$; otherwise, if $\D
\subseteq \C$ is a proper
non-degenerate fusion subcategory, then $\C \cong \D \boxtimes \Z_2(\D, \C)$
\cite[Theorem
3.13]{DGNOI},
and both $\D$ and $\Z_2(\D, \C)$ are solvable non-degenerate. By induction, $\D$
and $\Z_2(\D, \C)$ satisfy (i) or (ii), and
then so does $\C$.

The adjoint subcategory $\C_{ad}$ is a solvable braided fusion category and it
is in addition integral, by \cite[Proposition 8.27]{ENO}. If $\C_{ad} = \vect$,
then $\C$ is pointed and we are done.
We may assume that 
$\C_{ad} \ncong \vect$ and contains no nontrivial Tannakian subcategories
(otherwise $\C$ satisfies (i)).   By Proposition \ref{solv-tann}, we get that
$\C_{ad}$ is pointed, and therefore $\Z_2(\C_{ad})
\cong \svect$. Indeed, $\Z_2(\C_{ad})$ is pointed and symmetric, therefore it is
super-Tannakian, thus $\Z_2(\C_{ad})
\cong \svect$ in view of the assumption that $\C_{ad}$ contains no Tannakian
subcategories.

Hence
$\C_{ad}$ is slightly degenerate, and therefore $\C_{ad} \cong \svect \boxtimes
\C_0$, where $\C_0$ is a pointed non-degenerate braided fusion category
\cite[Proposition 2.6 (ii)]{ENO2}.
But $\C$ is prime, by assumption, and hence $\C_{ad} \cong \svect$. 

Since $\C$ is non-degenerate, then $\C_{ad} = \Z_2(\C_{pt}, \C)$ \cite[Corollary
3.27]{DGNOI}. Then we get $\C_{ad} = \Z_2(\C_{pt}, \C) \subseteq \C_{pt}$ and
thus
$\C_{ad} = \Z_2(\C_{pt}) \cong \svect$. Appealing again to \cite[Proposition 2.6
(ii)]{ENO2}, we obtain that $\C_{pt} = \C_{ad} = \svect$. Then, by \cite[Theorem
3.14]{DGNOI}, $\FPdim \C =
\FPdim \C_{ad} \FPdim \C_{pt} = 4$ and therefore $\C$ is an Ising braided
category. This finishes the proof of the theorem. 
\end{proof}

\section{The Witt class of a weakly group-theoretical non-degenerate braided
fusion category}\label{wgt}

Let $\W$ be the group of Witt classes of non-degenerate braided fusion
categories and let $\W_{pt}$ and $\W_{Ising}$ be the subgroups of Witt classes
of pointed non-degenerate fusion categories and Ising braided categories,
respectively. 

\medbreak If $\C$ is a non-degenerate braided fusion category, $\C$ is called
\emph{completely anisotropic} if the only connected \' etale algebra in $\C$ is
$A = \1$.
By \cite[Theorem 5.13]{witt-nondeg}, every non-degenerate braided fusion
category is Witt equivalent to a unique completely anisotropic non-degenerate
fusion category.

\begin{lemma}\label{wgt-ca} Let $\C$ be a weakly group-theoretical non-degenerate braided
fusion category. Suppose that $\C$ is completely anisotropic. Then $\C$ is
nilpotent.
\end{lemma}

Note that,  since every braided nilpotent fusion category is
solvable \cite[Proposition 4.5 (iii)]{ENO2}, it follows that $\C$ is also
solvable.

\begin{proof} 
By \cite[Corollary 3.8]{witt-braided}, equivalence classes of indecomposable
module categories over $\C$ are parameterized by isomorphism classes of triples
$(A_1, A_2, \phi)$,  where $A_1, A_2$ are connected \' etale algebras in $\C$
and $\phi:\C_{A_1}^0 \to (\C_{A_2}^0)^{\rev}$ is a braided equivalence.
Furthermore, invertible module categories correspond to such triples where $A_1
= A_2 = \1$ (see Remark 3.9 \textit{loc. cit.}). 

The assumption that $\C$ is weakly group-theoretical means that there exists an
indecomposable module
category   $\M$ such that
$\C^*_{\M}$ is nilpotent. Since $\C$ is braided,  $\M$ is naturally a
$\C$-bimodule category \cite[Section 2.8]{DN}. 
Consider the $\alpha$-induction tensor functors \cite[Section 5.1]{ostrik}
\begin{equation}\alpha^{\pm}:\C \to \C^*_{\M},
\end{equation}
defined by $\alpha^{\pm}(X)(M) = X \otimes M$, $X \in \C$, $M \in \M$. Letting 
$a_{X, Y, M}: X \otimes (Y \otimes M) \to (X \otimes Y) \otimes M$, $X, Y \in
\C$, $M \in \M$, denote the
associativi\-ty isomorphisms for the  $\C$-action on $\M$, the 
module functor structures on $\alpha^{\pm}(X)$, $X \in \C$, are given,
respectively, by 
$$a_{Y, X, M}^{-1} (c_{X, Y} \otimes \id) a_{X, Y, M}: \alpha^+(X)(Y \otimes M)
\to Y \otimes \alpha^+(X)(M),$$ and 
$$a_{Y, X, M}^{-1} (c^{\rev}_{X, Y} \otimes \id) a_{X, Y, M}: \alpha^-(X)(Y
\otimes M)
\to Y \otimes \alpha^-(X)(M),$$ $Y \in \C$, $M \in \M$.

The assumption that $\C$ is completely anisotropic implies that the module
category $\M$ is an invertible $\C$-bimodule category. Therefore the functors
$\alpha^{\pm}$ are equivalences of fusion
categories \cite[Proposition
4.2]{ENO3}. Hence $\C$ is nilpotent, as claimed.  \end{proof}

\begin{theorem}\label{wgt-witt} Let $\C$ be a non-degenerate braided fusion
category. Suppose that $\C$ is weakly group-theoretical. Then $[\C] \in \langle
\W_{pt}, \W_{Ising}\rangle$. If in addition $\C$ is integral, then $[\C] \in
\W_{pt}$.
\end{theorem}

\begin{proof} The proof is by induction on
$\FPdim \C$. We may assume that $\C$ is prime. If $A \in \C$ is a connected \'
etale algebra, then $\C$ is Witt equivalent to the non-degenerate fusion
category $\C_{A}^0$. Moreover, $\C_{A}^0$ is also weakly group-theoretical,
since it is a fusion subcategory of a quotient category of $\C$, and $\FPdim
\C_{A}^0 = \FPdim \C / (\FPdim A)^2$.  Hence we may assume that $\C$ is
completely anisotropic, otherwise the statement follows by induction. 
By Lemma \ref{wgt-ca}, we get that $\C$ is solvable. In particular, being
completely anisotropic, $\C$ contains no nontrivial Tannakian subcategory and it
follows from  Theorem \ref{sol-non-deg} that $\C
\cong \mathcal B \boxtimes \I_1 \boxtimes \dots \boxtimes \I_n$,
as braided fusion categories, where $\B$ is a pointed non-degenerate fusion
category and $\I_1, \dots, \I_n$ are Ising braided categories.   Hence $[\C] \in
\langle
\W_{pt}, \W_{Ising}\rangle$. 
Moreover, if $\C$ is integral, then so is $\C^0_A$, hence we may also assume
inductively that $\C$ is completely anisotropic.    
Then Theorem \ref{sol-non-deg} implies that $\C$ is indeed pointed in this case. 
This finishes the proof of the theorem.
\end{proof}

Let $\widetilde\W$ denote the subgroup generated by Witt equivalence classes of 
the
fusion categories $\C(\mathfrak g, k)$ of integrable highest weight modules of
level $l$ over the affinization of a simple finite-dimensional Lie algebra
$\mathfrak g$. 

By \cite[Remark 6.5]{witt-nondeg}, $\W_{pt} \subseteq \widetilde\W$. On the
other hand, for any Ising braided category $\I$, we have $[\I] =
[\C(\mathfrak{sl}(2),
2)]^m$, for a unique odd number $m$, $1\leq m \leq 15$ \cite[Section 6.4
(3)]{witt-nondeg}. Thus the subgroup generated by $\W_{pt}$ and
$\W_{Ising}$ is contained in $\widetilde\W$.
As a consequence of Theorem \ref{wgt-witt}, we get:

\begin{corollary} Let $\C$ be a non-degenerate braided fusion category. Suppose
that $\C$ is weakly group-theoretical. Then $[\C] \in \widetilde\W$.
\end{corollary}

\begin{remark} Let $s\W$ denote the Witt group of slightly degenerate braided
fusion categories
introduced in \cite{witt-braided}.
Recall from \textit{loc. cit.} that there is a group homomorphism $S: \W \to
s\W$, defined by $S([\C]) = [\C \boxtimes \svect]$, whose kernel is the subgroup
of $\W$ generated by the Witt classes of Ising braided categories.
It follows from Theorem \ref{wgt-witt} that for every weakly group-theoretical
non-degenerate braided fusion
category $\C$, we have $S([\C]) \in s\W_{pt}$. 
\end{remark}

We also point out the following consequence of Theorem \ref{wgt-witt}:

\begin{corollary} Let $\C$ be an integral non-degenerate braided fusion
category. Suppose $\C$ is weakly group-theoretical (respectively, solvable).
Then there exist an integral
nilpotent (respectively, cyclically nilpotent) fusion
category $\D$ and a pointed non-degenerate completely anisotropic fusion
category $\B$ such that
$\Z(\D) \cong \C \boxtimes \B$ as braided fusion categories.
\end{corollary}

\begin{proof} By Theorem \ref{wgt-witt}, there exist fusion category $\D$
and a pointed non-degenerate fusion category $\B$ such
that $\Z(\D) \cong \C \boxtimes \B$ as braided fusion categories. 
Moreover, since every Witt class has a unique representative which is completely
anisotropic, we may assume that $\B$ is completely anisotropic.

This implies that $\Z(\D)$ is
integral and weakly group-theoretical, and therefore so is $\D$. Hence there
exists an indecomposable module category $\M$ such that $\D^*_{\M}$ is
nilpotent. Furthermore, if $\C$ is solvable, then so is $\Z(\D)$, and therefore
there exists $\M$ such that $\D^*_{\M}$ is cyclically
nilpotent. This implies the corollary, since $\Z(\D) \cong \Z(\D^*_{\M})$ as
braided tensor categories.
\end{proof}

\section{Sufficient conditions for a non-degenerate braided fusion category to
be weakly group-theoretical}\label{suff-cond}

Let $\C$ be a fusion category. Let $\Irr(\C)$ be the set of isomorphism classes
of simple objects of $\C$ and let $G(\C)$ be the group of isomorphism classes of
invertible objects. 

The group $G(\C)$ acts on
the set of isomorphism classes of simple objects by tensor multiplication. For a
simple object 
$X \in \C$ let $G[X]$ denote the stabilizer of $X$ under this action. Thus
$G[X]$ is a subgroup of $G(\C)$ of order dividing $(\FPdim X)^2$. Moreover, for
every $X \in \Irr(\C)$, we have an
isomorphism
\begin{equation}\label{xx*}X \otimes X^* \cong \bigoplus_{g \in G[X]}g \oplus
\bigoplus_{\substack{Y \in \Irr(\C)\\ \FPdim Y > 1}} \Hom_{\C}(Y, X \otimes X^*)
\otimes Y.\end{equation}

\begin{lemma}\label{sl-deg} Let $\C$ be a braided fusion category. Suppose that
$\C$ contains
no nontrivial non-degenerate or Tannakian fusion subcategories. Then $\C$ is
slightly degenerate and the following hold:
\begin{enumerate}
 \item[(i)] $\C_{pt} = \Z_2(\C) \cong \svect$;
\item[(ii)] $G[X] = \1$, for all simple object  $X \in \C$. 
\end{enumerate}
\end{lemma}

\begin{proof} Consider the M\" uger center $\Z_2(\C)$ of $\C$. Then $\Z_2(\C)$
is a symmetric fusion subcategory and therefore it is super-Tannakian. 
The assumptions on $\C$ imply that $\Z_2(\C) \ncong \vect$ and also
that $\Z_2(\C)$ contains no nontrivial
Tannakian subcategories. Hence $\Z_2(\C) \cong \svect$ \cite[Section
2.4]{ENO2} and therefore $\C$ is slightly degenerate. Note that $\Z_2(\C)
\subseteq \C_{pt}$.
Moreover, since $\C_{pt}$ cannot contain any nontrivial non-degenerate or
Tannakian fusion
subcategory, then it is slightly degenerate as well.
By \cite[Proposition 2.6 (ii)]{ENO2}, every slightly degenerate pointed braided
fusion category factorizes in the form $\svect \boxtimes \B$, where $\B$ is a
pointed non-degenerate braided fusion category.  This implies that in our case
$\C_{pt} = \Z_2(\C) \cong
\svect$, whence we get part (i).
Let $\1 \neq g \in \svect$ be the unique nontrivial (fermionic)
invertible object. If $X \in \C$ is a simple
object, we have $g \otimes X \ncong X$ \cite[Proposition 2.6 (i)]{ENO2}.  This
implies part (ii), since by (i), $g$ is the only nontrivial invertible object of
$\C$.
\end{proof}

It is well-known that if all the character degrees of a finite group $G$ are
powers of a prime number $p$, then $G$ is solvable \cite{isaacs}.
The following theorem extends this result to braided fusion categories. Some
instances of the theorem were obtained previously in \cite[Theorem
7.3]{char-deg}
and \cite[Theorem 1.1]{cd2-wint}. 

\begin{theorem}\label{fpdim-pn} Let $\C$ be a braided fusion category such that
$\FPdim \C \in
\mathbb Z$. 
Suppose that $p$ is a prime number such that $\FPdim X$ is a power of $p$, for
all simple object $X \in \C$.
Then $\C$ is solvable.  
\end{theorem}

Note that since the Frobenius-Perron dimension of $\C$ is an integer, we have
$(\FPdim X)^2 \in \mathbb Z$, for all $X \in \Irr(\C)$ \cite[Proposition
8.27]{ENO}. Therefore the possible
powers of $p$ that can occur as simple Frobenius-Perron dimensions in $\C$ are
half-integer powers.

\begin{proof} The proof is by induction on $\FPdim \C$. We may assume that $\C$
is integral. Otherwise, $\C$ is a $U(\C)$-extension of its integral fusion
subcategory $\C_{ad}$, where $U(\C)$ denotes the universal grading group of
$\C$. By induction, $\C_{ad}$ is solvable. Since $\C$ is braided, its universal
grading group is abelian, and therefore $\C$ is also solvable.  

It will be enough to
show that $\C$ contains a nontrivial Tannakian
subcategory $\E$. In such case, $\E \cong \Rep G$ for some finite group $G$ and
$G$
is solvable, because $\dim Y = p^m$, $m \geq 0$, for all simple objects $Y \in
\Rep G$. 
Moreover, it follows from \cite[Corollary 2.13]{fusionrules-equiv}, that the
Frobenius-Perron dimensions of simple objects in the
de-equivariantization $\C_G$,  and thus also in its fusion subcategory $\C_G^0$,
are powers of $p$ as well. Since $\FPdim \C^0_G \leq \FPdim \C_G  = \FPdim \C /
|G| < \FPdim \C$
and $\C_G^0$ is braided, then $\C_G^0$ is solvable, by
induction. Hence so is $\C$, by Proposition \ref{tann-inher}.

In view of the relations \eqref{xx*}, the assumption implies that for every
simple object $X$ of $\C$ the order of the group $G[X]$ is divisible by $p$.
If $\C$ contains no nontrivial non-degenerate or
Tannakian subcategories, then Lemma \ref{sl-deg} applies, and we obtain that
$G[X] = \1$, for all simple object $X \in \C$, which is a contradiction.

We may thus assume that $\C$ contains a nontrivial non-degenerate fusion
subcategory.
Suppose first that $\C$ is itself non-degenerate. Since $p$ divides  $\FPdim
\C_{pt}$, then $\C_{pt} \neq \vect$.  
 Hence  $\C_{ad} = (\C_{pt})' \subsetneq
\C$ and, by induction, $\C_{ad}$ is solvable.  Then so is $\C$, because it is a
$U(\C)$-extension of $\C_{ad}$ and $U(\C)$ is abelian. 
If, on the other hand, $\D \subsetneq \C$ is a nontrivial non-degenerate fusion
subcategory, then $\C \cong \D \boxtimes \D'$ and $\FPdim \D, \FPdim \D' <
\FPdim \C$.
Hence $\D$ and $\D'$ are both solvable by induction and therefore so is $\C$.
This completes the proof of the theorem.
\end{proof}

\begin{corollary} Let $\C$ be a non-degenerate braided fusion category and let
$p$ be a prime number. Suppose that $\FPdim \C = p^a c$, where $a \geq 0$ is an integer, and
$c$ is a square free natural number. Then $\C$ is solvable.
\end{corollary}

\begin{proof} Let $X \in \C$ be a simple object. Since $\C$ is non-degenerate,
then $(\FPdim X)^2$ divides $\FPdim \C$ \cite[Theorem 2.11 (i)]{ENO2}.
If $\C$ is integral,  then $\FPdim X$ must be a power of $p$ for all $X \in
\Irr(\C)$ and therefore $\C$ is solvable, by Theorem \ref{fpdim-pn}.
Suppose next that $\C$ is not integral. Then $\C$ is a $G$-extension of an
integral fusion subcategory $\D$, where $G$ is an elementary abelian $2$-group
\cite[Theorem 3.10]{gel-nik}. Again in this case, we get that the
Frobenius-Perron dimension of a simple object of $\D$ is a power of $p$ and
therefore the braided fusion category $\D$
is solvable, by Theorem \ref{fpdim-pn}. Then $\C$, being a $G$-extension of
$\D$,
is also solvable. 
\end{proof}

\begin{theorem}\label{fact-wgt} Let $p$ and $q$ be prime numbers. Let $\C$ be a
non-degenerate braided fusion category such that $\FPdim \C = p^aq^bc$, where
$a$ and $b$ are nonnegative integers, and $c$ is  a square-free natural number.
Then $\C$ is weakly group-theoretical.
\end{theorem}

\begin{proof} Observe that, after
eventually replacing $c$ by an appropriate divisor, we may assume that $c$ is
relatively
prime to $p$ and $q$. The proof of the theorem is by induction on $\FPdim \C$. 
As in the proof of Theorem \ref{fpdim-pn}, we may assume that $\C$ is integral
and it will be enough to show that $\C$ contains a nontrivial Tannakian
subcategory. 

Let us assume that $\C$ is not nilpotent (and in particular it is not pointed), otherwise there is nothing
to prove. Since $\C$ is non-degenerate, then for every simple object $X \in \C$,
we have that $(\FPdim X)^2$ divides $\FPdim \C$.
Hence $a \geq 2$ or $b \geq 2$ and moreover, for every simple object $X$, we
have $\FPdim X = p^nq^m$, for some $n, m \geq 0$.

Suppose first that $\C$ has no non-invertible simple object of prime power
dimension. Then $pq|\FPdim X$, for all non-invertible $X \in \Irr(\C)$.
In view of the relations \eqref{xx*}, this implies that for any fusion
subcategory $\D$, such that $\D$ is not pointed, the Frobenius-Perron
dimension of $\D_{pt} = \D \cap \C_{pt}$ is divisible by $pq$.
In particular, $\FPdim \C_{pt} \cap \C_{ad}$ is divisible by $pq$ and
thus it is bigger than $2$. But, since $\C$ is non-degenerate, then
$\C_{pt} = \C_{ad}'$ and therefore the category $\C_{pt} \cap \C_{ad}$ is
symmetric. It follows that $\C_{pt} \cap \C_{ad}$ contains a nontrivial
Tannakian subcategory and we are done. 

Suppose next that $\C$ has a simple object of positive prime power dimension. By
\cite[Corollary 7.2]{ENO2}, $\C$
contains a nontrivial symmetric subcategory $\D$.
We may assume that $\D$ contains no nontrivial Tannakian subcategory, and thus
$\D \cong \svect$. Since $\D'' = \D$, then $\D'$ is a slightly degenerate
fusion category.

If $\D'$ has a simple object of odd prime power dimension, then it contains a
nontrivial Tannakian subcategory by \cite[Proposition 7.4]{ENO2}, and we are
done. 
If $\FPdim X$ is divisible by $pq$ for all non-invertible simple object $X \in \D'$, then $pq$
divides the order of the group $G[X]$ for all non-invertible $X \in \Irr(\D')$, by
\eqref{xx*}. 
In view of Lemma \ref{sl-deg}, we may assume that $\D'$ contains a nontrivial
non-degenerate fusion subcategory $\B$.  
Then $\C \cong \B \boxtimes \B'$, where $\B$ and $\B'$ are both non-degenerate.
Then $\FPdim \B \FPdim \B' = \FPdim \C = p^aq^bc$ and  $\FPdim \B, \FPdim \B' <
\FPdim \C$. It follows by induction that $\B$ and $\B'$ are both weakly
group-theoretical and then so is $\C$.

It remains to consider the case where $\FPdim X = 2^m$, $m \geq 0$, for every
simple object $X$ of $\D'$. In this case, Theorem \ref{fpdim-pn} implies that
$\D'$ is solvable.
Then it follows from Proposition \ref{solv-tann}, that either $\D'$ contains a
nontrivial
Tannakian subcategory, in which case we are done, or $\D'$ is pointed.  Suppose
that $\D'$ is pointed. By
\cite[Proposition 2.6 (ii)]{ENO2}, $\D' \cong \svect \boxtimes \B$, where $\B$
is a pointed non-degenerate fusion category. 
If $\B$ is not trivial then, as before, $\C \cong \B \boxtimes \B'$, where $\B$
and $\B'$ are both non-degenerate and  $\FPdim \B, \FPdim \B' <
\FPdim \C$, hence $\C$ is weakly
group-theoretical, by induction.
If, on the other hand, $\B \cong \vect$, then $\FPdim \D' = 2$ and therefore
$\FPdim \C = \FPdim \D \FPdim \D' = 4$. Hence $\C$ is nilpotent and in
particular it is weakly group-theoretical as well. 
This completes the proof of the theorem. 
\end{proof}

\section{Non-degenerate braided fusion categories of low dimension}\label{low-dim}

As an application of Theorem \ref{fact-wgt} we prove in this section that
non-degenerate fusion categories of small dimension are weakly
group-theoretical.

\begin{theorem}\label{less1800} Let $\C$ be a weakly integral non-degenerate
fusion category
such that $\FPdim \C < 1800$. Then $\C$ is weakly group-theoretical.
 \end{theorem}

\begin{proof} Every natural number $n < 1800$ such that $n \neq 900$, factorizes
in the form $n = p^aq^bc$, where $p$ and $q$ are prime numbers, $a, b \geq 0$,
and $c$ is a square-free integer. 
In view of Theorem \ref{fact-wgt} it will be enough to consider the case where
$\FPdim \C = 900$.

We may assume that $\C$ is a prime non-degenerate fusion category, that is, $\C$
contains no nontrivial proper non-degenerate fusion subcategory, and in addition
$\C$ contains no nontrivial Tannakian subcategory. Indeed, if $\D \subseteq \C$
is a nontrivial proper non-degenerate fusion subcategory, then $\C \cong \D
\boxtimes \D'$ where $\D$ and $\D'$ are non-degenerate fusion subcategories of
Frobenius-Perron dimension strictly less than $900$, and thus weakly
group-theoretical. Then $\C$ is weakly group-theoretical in this case.  
Similary, if $\C$ contains a nontrivial Tannakian subcategory $\E \cong \Rep G$,
where $G$ is a finite group, $|G| > 1$, then the de-equivariantization $\C_G$ is
a $G$-crossed braided fusion category of Frobenius-Perron dimension strictly
less than $900$, whose neutral component $\C_G^0$ is non-degenerate and thus
weakly group-theoretical. Hence $\C_G$ and $\C$ are both weakly
group-theoretical as well.  

It follows from the proof of \cite[Theorem 9.2]{ENO2} that a non-degenerate
integral fusion category of Frobenius-Perron dimension $p^2q^2r^2$, where $p <
q< r$ are prime numbers, contains a nontrivial Tannakian subcategory. 
Hence we may assume that $\C$ is not integral. 

Therefore $\C$ is
an $E$-extension of an integral fusion subcategory $\D$, where $E$ is an
elementary abelian $2$-group \cite[Theorem 3.10]{gel-nik}. Then $|E| = 2$ or $4$
and $\FPdim \D = \FPdim \C / |E|$. 
Hence we may assume $\FPdim \D = \FPdim \C / 2$, because otherwise $\D$ and
therefore also $\C$ would be solvable, in view of \cite[Theorem 1.6]{ENO2}.
We may further assume that $\D$ contains no nontrivial non-degenerate or
Tannakian fusion subcategories. 
It follows from Lemma \ref{sl-deg} that  $\D$ is slightly degenerate, $\D_{pt}
\cong \svect$ and $G[X] = \1$, for all simple object $X \in \D$.

In addition, if $X \in \D$ is a simple object, then $(\FPdim X)^2$ divides
$900 = \FPdim \C$. Thus $\FPdim X = 1, 2, 3, 5, 6, 10$ or $15$. Since the
group of invertible objects of $\D$ is of order $2$ and $G[X] = \1$, then the
number of simple objects of $\D$ of a given Frobenius-Perron dimension must be
even. In
particular, since $\FPdim \D = 2 (15)^2$, then $\D$ cannot have simple objects
of Frobenius-Perron dimension $15$. Also, by \cite[Proposition 7.4]{ENO2}, $\D$
has no simple objects of Frobenius-Perron dimension $3$ or $5$.
Thus we conclude that the Frobenius-Perron dimension of every non-invertible
simple object $X$ of $\D$ is necessarily even. Decomposing $X \otimes X^*$ into
a sum of simple objects and using that $G[X] = \1$ we arrive to a contradiction;
see \eqref{xx*}.
This shows that $\C$ is weakly group-theoretical, as claimed.  
\end{proof}

The result in Theorem \ref{less1800} can be strengthened in the odd-dimensional
case. 
In fact, we have:

\begin{theorem}\label{oddless33075} Let $\C$ be a weakly integral non-degenerate
fusion category
such that $\FPdim \C$ is odd and $\FPdim \C < 33075$. Then $\C$ is solvable.
\end{theorem}

\begin{proof} It will be enough to show that $\C$ is weakly group-theoretical,
since any odd-dimensional weakly group-theoretical fusion category is
necessarily solvable \cite[Proposition 7.1]{char-deg}. 
The assumption that $\FPdim \C$ is odd implies furthermore that $\C$ is integral
\cite[Corollary 3.11]{gel-nik}.

Observe that an odd natural number $n < 33075 = 3^3 5^2 7^2$ such that $n \neq
11025$, factorizes
in the form $n = p^aq^bc$, where $p$ and $q$ are prime numbers, $a, b \geq 0$,
and $c$ is a square-free integer. By Theorem \ref{fact-wgt}, we only need to
consider the case where
$\FPdim \C = 11025 = 3^2 5^2 7^2$. In this case, it follows from the proof of
\cite[Lemma 9.3]{ENO2} that $\C$ contains a nontrivial symmetric (thus
Tannakian)
subcategory $\E \cong \Rep  G$. Then $\C_G^0$ is weakly group-theoretical and
hence so is $\C$. 
\end{proof}

\bibliographystyle{amsalpha}

\end{document}